\documentclass{amsart}

\usepackage{amsmath}
\usepackage{amssymb}
\usepackage{amsfonts}
\usepackage{amsopn}
\usepackage{amsthm}
\usepackage[all]{xy} 
\usepackage{hyperref}

\swapnumbers
\theoremstyle{plain}
\newtheorem{thm}{Theorem}[section]
\newtheorem{lem}[thm]{Lemma}
\newtheorem{cor}[thm]{Corollary}
\newtheorem{prop}[thm]{Proposition}

\theoremstyle{definition}

\newtheorem{ex}[thm]{Example}

\DeclareMathOperator{\CH}{CH}  
\DeclareMathOperator{\K0}{K_0}   
\DeclareMathOperator{\codim}{codim} 
\DeclareMathOperator{\im}{im} 
\DeclareMathOperator{\ind}{ind} 
\DeclareMathOperator{\id}{id} 

\newcommand{\XX}{{}_\xi X}  

\newcommand{\ZZ}{\mathbb{Z}}   

\newcommand{\cc}{\mathfrak{c}}  

\newcommand{\Fp}{\mathbb{F}_p}  

\newcommand{\EE}{\mathcal{E}}  

\newcommand{\LL}{\mathcal{L}}   

\DeclareMathOperator{\resg}{\it{res}_\gamma}   

\DeclareMathOperator{\resch}{\it{res}_{CH}}   

\DeclareMathOperator{\resk}{\it{res}_{k'/k}}   

\DeclareMathOperator{\resl}{\it{res}_{l/k}}  

\DeclareMathOperator{\itau}{\tau^{{\it i/i}+1}K_0(X)}  

\DeclareMathOperator{\igam}{\gamma^{{\it i/i}+1}K_0(X)}  

\DeclareMathOperator{\igamtwist}{\gamma^{{\it i/i}+1}K_0({}_\xi X)}  

\DeclareMathOperator{\citau}{c_{\it i}^\tau}  

\DeclareMathOperator{\cigam}{c_{\it i}^\gamma}  

\title{The J-invariant and Tits algebras for groups of inner type $E_6$}

\author{Caroline Junkins}

\begin{document}

\begin{abstract}
A connection between the indices of the Tits algebras of a split linear algebraic group $G$ and the degree one parameters of its motivic 
$J$-invariant was introduced by Qu\'eguiner-Mathieu, Semenov and Zainoulline through use of the second Chern class map in the Riemann-Roch theorem without denominators. 
In this paper we extend their result to higher Chern class maps and provide applications to groups of inner type $E_6$. 
\end{abstract}

\maketitle


\section*{Introduction}

For a reductive group $G$, invariants known as the \textit{Tits algebras} were introduced by J. Tits in \cite{Ti71} and have proven to be an invaluable tool for the computation of the $K$-theory of twisted flag varieties by Panin \cite{Pa94} and for the index reduction formulas by Merkurjev, Panin and Wadsworth \cite{MPW}. Furthermore, the Tits algebras are essentially the only cohomologogical invariant of degree 2 (cf. \cite[Ex. 31.21]{INV}, and have applications to both the classification of linear algebraic groups and the study of the associated homogeneous varieties. \\
\indent 
The $J$-invariant, as defined by Petrov, Semenov and Zainoulline in \cite{PSZ}, is an invariant of $G$ which describes the motivic behaviour of the variety of Borel subgroups of $G$. For a prime $p$, the $J$-invariant of $G$ modulo $p$ is given by a set of integers $J_p(G)=(j_1,\dots,j_r)$. We consider also a (possibly empty) subset $J_p^{(1)}(G)=(j_1,\dots,j_s)$, $s\leq r$, consisting of the parameters of the $J$-invariant of degree 1. \\
\indent 
Motivated by the work \cite{GZ10}, Qu\'eguiner-Mathieu, Semenov and Zainoulline discovered a connection 
between these degree 1 parameters and the indices of the Tits algebras of $G$. This connection is developed in \cite{QSZ}, through use of the second Chern class map in the Riemann-Roch theorem without denominators. 
The goal of this paper is to extend their result through use of higher Chern class maps (cf. Theorem~\ref{mainthm}). 
We then apply this result to a group $G$ of inner type $E_6$. We provide an explicit connection between the values of $J_3^{(1)}(G)$ and the index of the Tits algebra of $G$ (see Proposition~\ref{maincor}). This result is used by Garibaldi, Petrov and Semenov to give finer information, summarized in their Table 10A, on the $J$-invariant for groups of type $E_6$ \cite{GPS}.

This paper is organized as follows. In the first section, we review the notion of characteristic classes and introduce the topological and $\gamma$-filtrations on the Grothendieck group $\K0$ of a smooth projective variety. In Section 2, we look more closely at the $\gamma$-filtration of the variety $X=G/B$ of Borel subgroups, and give a simplified definition through use of the Steinberg basis. We then consider a twisted form $\XX$ of $X$ by means of a $G$-torsor $\xi\in Z^1(k, G)$ and the $\gamma$-filtration of $\K0(\XX)$. In Section 3, we recall the definition of the Tits map and its relation to our primary objects of concern, the common index $i_c$ and the $J$-invariant of $\xi$. Section 4 provides the main result of the paper, a relationship between the common index and the possible values of the indices of the $J$-invariant of degree 1. In the final section, we give an application of this result to a group of inner type $E_6$.

\noindent{\small{\textbf{Acknowledgments.}  
I am very grateful to my PhD supervisor K. Zainoulline for useful discussions on the subject of this paper.
This work has been partially supported by his NSERC Discovery grant 385795-2010.
The last part of these notes has been written during my research visit to the Johannes Gutenberg University of Mainz.
My special gratitude is due to N. Semenov and the University of Mainz for their hospitality and support. Finally, I would like to thank B. Calm\`es for providing many helpful suggestions on the first draft of these notes. 
}}


\section{Gamma filtration on \texorpdfstring{$\K0$}{K0} and the Chern class map}

In the present section we recall several useful properties of the topological filtration
and the $\gamma$-filtration on Grothendieck's $K_0$ of a smooth projective variety.
The reader is advised to consult \cite{SGA6}, \cite{FL} and \cite[ch.15]{Fu} for more details.

Let $X$ be a smooth projective variety over a field $k$.
Consider the topological filtration on $\K0(X)$ (cf. \cite[Ex.~15.1.5]{Fu}) given by 
$$
\tau^i\K0(X) =\langle [{\mathcal O}_V]\mid \codim\; V\geq i \rangle,
$$
where ${\mathcal O}_V$ is the structure sheaf of a closed subvariety $V$ in $X$. 
We denote by $\itau$, $i\ge 0$ the degree $i$ component $\tau^i\K0(X)/\tau^{i+1}\K0(X)$ of the corresponding graded ring.
There is a surjection
$$
p\colon \CH^i(X) \twoheadrightarrow \itau,\quad V\mapsto[{\mathcal O}_V],
$$
from the Chow group of codimension $i$ cycles.

For a vector bundle $\EE$ on $X$, the total Chern class $c(\EE)=1+c_1(\EE)t+ c_2(\EE)t^2+ \dots$ is an element of $\CH(X)[t]$, and by the Whitney sum formula, it defines a group homomorphism
$$
c:\K0(X) \to \CH(X)
$$
For $\alpha\in\K0(X)$, $c_i(\alpha)$ is the component of $c(\alpha)$ in $\CH^i(X)$, giving a group homomorphism
$$
c_i\colon \tau^i\K0(X) \to \CH^i(X)
$$
defined by taking the $i$-th Chern class.

\begin{lem} \label{lem:prelims}
For a smooth projective variety $X$ over a field $k$, $c_i(\tau^j\K0(X))=0$ for all $0<i<j$, and the induced homomorphism 
$$
\citau\colon \itau\to \CH^i(X)
$$
is an isomorphism over the coefficient ring $\ZZ[\tfrac{1}{(i-1)!}]$ for all $i>0$. 
\end{lem}

\begin{proof}
From \cite[Ex 15.3.6]{Fu} it can be seen that $c_{i}(\tau^{i+1}\K0(X))=0$ for all $i\in\ZZ_{>0}$, and by the definiton of the topological filtration, $\tau^j\K0(X)\subseteq \tau^i\K0(X)$ for all $j>i$. Thus, $c_i(\tau^j\K0(X))=0$ for all $0<i<j$.

In \cite[Ex. 15.3.6]{Fu} it is shown that the composite $\citau\circ p$ is simply multiplication by $(-1)^{i-1}(i-1)!$. This result implies that $\citau$ is an isomorphism for $i\leq 2$ and, moreover, the map $\citau$ is an isomorphism over the coefficient ring $\ZZ[\tfrac{1}{(i-1)!}]$.
\end{proof}

For any $x\in\K0(X)$, let $\gamma(x)=1+\gamma_1(x)t+\gamma_2(x)t^2+\dots$ denote the total characteristic class of $x$
with values in $K_0$ as defined in \cite[Ex. 3.2.7(b)]{Fu}. We follow the convention $\gamma_1([\LL])=1-[{\LL^\vee}]$ for any line bundle $\LL$ over $X$, where $\LL^\vee$ denotes the dual of $\LL$.

\begin{ex}
Using the Whitney sum formula we obtain the following results
for $i=1,2$ by computing the total Chern classes,
$$
c(\gamma_1([\LL]))= c(1-[\LL^\vee])=\frac{1}{1-c_1(\LL)t}=1+c_1(\LL)t+c_1(\LL)^2t^2+\dots
$$
This gives $c_1(\gamma_1([\LL]))=c_1(\LL)$.
Similarly,
$$
c(\gamma_1([\LL_1])\gamma_1([\LL_2]))=c((1-[\LL_1^\vee])(1-[\LL_2^\vee]))=
\frac{c(1)\cdot c([(\LL_1\otimes \LL_2)^\vee])}{c([\LL_1^\vee])\cdot c([\LL_2^\vee])}
$$
$$
=(1-(c_1(\LL_1)+c_1(\LL_2))t)(1+c_1(\LL_1)t+c_1(\LL_1)^2t^2+\dots)(1+c_1(\LL_2)t+c_1(\LL_2)^2t^2+\dots)
$$
Hence $c_2(\gamma_1([\LL_1])\gamma_1([\LL_2]))=-c_1(\LL_1)c_1(\LL_2)$.

By the definition of these characteristic classes (cf. \cite[Ex 15.3.6]{Fu}), we have in general
\begin{equation} \label{firsteq}
c_i\bigl(\gamma_1([\LL_1])\ldots \gamma_1([\LL_i])\bigr)=(-1)^{i-1}(i-1)!\cdot c_1(\LL_1)\cdot \ldots\cdot c_1(\LL_i).
\end{equation}
\end{ex}

The Grothendieck $\gamma$-filtration on 
$\K0(X)$ is defined by
$$
\gamma^i\K0(X)=
\langle \gamma_{i_1}(x_1)\cdot \ldots \cdot \gamma_{i_m}(x_m) \mid
i_1+\ldots + i_m\ge i,\; x_l\in \K0(X)\rangle,
$$
(cf. \cite[Ex.15.3.6]{Fu}, \cite[Ch.3 and 5]{FL}).  
Let $\igam$ denote the degree $i$ component $\gamma^i\K0(X)/\gamma^{i+1}\K0(X)$ of the corresponding graded ring. 

It is known that $\gamma^i\K0(X)$ is contained in $\tau^i\K0(X)$ for every $i\ge 0$, 
and they coincide for $i\leq 2$  (cf. \cite[Prop.2.14]{Ka98}). Therefore, by Lemma \ref{lem:prelims}
the Chern class map $c_i$ restricted to $\gamma^i\K0(X)$ vanishes on $\gamma^{i+1}\K0(X)$, and induces a map 
$$
\cigam\colon \igam\to \CH^i(X). 
$$

\begin{ex}
For $i=1$ we have $\gamma^{1/2}\K0(X)=\tau^{1/2}\K0(X)$ and $c_1\circ p=\id_{\CH^1(X)}$, giving an isomorphism 
$$
c_1\colon \gamma^{1/2}\K0(X)\tilde{\to}\CH^1(X).
$$

In the previous example, we saw that the map $c_1$ sends $\gamma_1([L])$ to $c_1(L)$. For $i=2$ we again have $\gamma^2\K0(X)=\tau^2\K0(X)$, but this time $\gamma^3\K0(X)$ does not necessarily coincide with $\tau^3\K0(X)$. We may form an exact sequence,
$$
0\to \tau^3\K0(X)/\gamma^3\K0(X)\to \gamma^{2/3}\K0(X) \to \tau^{2/3}\K0(X)\to 0.
$$
Replacing $\tau^{2/3}\K0(X)$ with $\CH^2(X)$, the map $c_2:\gamma^{2/3}\K0(X)\to \CH^2(X)$ is surjective. In addition, $\ker(c_2)\cong \tau^3\K0(X)/\gamma^3\K0(X)$. Note that for all $i\ge0$, (cf. \cite[Prop. 2.14]{Ka98}) $\tau^i\K0(X)\otimes\mathbb{Q}\cong\gamma^i\K0(X)\otimes\mathbb{Q}$, thus $\ker(c_2)$ is torsion. 
\end{ex}

\begin{prop} \label{propsurj}
The Chern class map $\cigam\colon \igam\to \CH^i(X)$
is surjective over the coefficient ring $\ZZ[\tfrac{1}{(i-1)!}]$.
\end{prop}

\begin{proof}
By Lemma \ref{lem:prelims}, the map $\citau:\tau^{i/{i+1}}(X)\to \CH^i(X)$ is surjective over the coefficient ring $\ZZ[\tfrac{1}{(i-1)!}]$. Since $\gamma^i\K0(X)\subseteq \tau^i\K0(X)$ for all $i$, we have an obvious map $\gamma^{i/{i+1}}\K0(X)\to \tau^{i/{i+1}}\K0(X)$ defined by sending $x+\gamma^{i+1}\K0(X) \mapsto x+\tau^{i+1}\K0(X)$. By definition, $\cigam$ is the composition of these two maps. 

\begin{displaymath}
    \xymatrix{
        \gamma^{i+1}\K0(X) \ar[r] \ar@/_/@<-1ex>[rr]_{\cigam} &\tau^{i/{i+1}}\K0(X)  \ar@{->>}[r]^\citau &\CH^i(X)  }
\end{displaymath}

\noindent  Thus $\im(\cigam)\subseteq \im(\citau)$, and so it remains to show that $\im(\citau)\subseteq \im(\cigam)$ over $\ZZ[\tfrac{1}{(i-1)!}]$, for all $i$. 

Consider an arbitrary element $x\in\K0(X)$. By the splitting principle we may assume that $x=\LL_1\oplus\dots\oplus \LL_n$ where $\LL_1,\dots,\LL_n$ are line bundles over $X$. By the properties of characteristic classes, $\gamma_i(\LL_1\oplus\dots\oplus \LL_n)=0$ for all $i> n$ and $\gamma_i(\LL_1\oplus\dots\oplus \LL_n)=s_i(\gamma_1(\LL_1),\dots,\gamma_1(\LL_n))$ for all $0<i\le n$, where $s_i(\gamma_1(\LL_1),\dots,\gamma_1(\LL_n))$ is the $i$-th elementary symmetric polynomial in variables $\gamma_1(\LL_1),\dots,\gamma_1(\LL_n)$. Thus, taking the total Chern class, we have by \eqref{firsteq} and Lemma \ref{lem:prelims}
\begin{align*}
c(\gamma_i(\LL_1\oplus\dots\oplus \LL_n))&=c(s_i(\gamma_1(\LL_1),\dots,\gamma_1(\LL_n)))\\
			&=\prod_{1\leq j_1<\dots<j_i\leq n}(1+(-1)^{i-1}(i-1)!c_1(\LL_{j_i})\cdots c_1(\LL_{j_i})t^i+\dots)\\
			&=1+(-1)^{i-1}(i-1)!\cdot s_i(c_1(\LL_1),\dots,c_1(\LL_n))t^i+\dots
\end{align*}
Thus, $c_i(\gamma_i(\LL_1\oplus\dots\oplus \LL_n))=(-1)^{i-1}(i-1)! s_i(c_1(\LL_1),\dots,c_1(\LL_n))$. In general,
\begin{equation} \label{gammatoc}
c_i(\gamma_i(x))=(-1)^{i-1}(i-1)! c_i(x)\text{ for all }x\in\K0(X).
\end{equation}
\end{proof}


\section{Gamma filtration on twisted flag varieties}

In the present section we discuss the $\gamma$-filtration on the variety of Borel subgroups. A reader is encouraged to consult \cite{QSZ} and \cite{GZ10} for further details.
\par
Let $G$ be a split simple linear algebraic group of rank $n$ over a field $k$. We fix a split maximal torus $T$
and a Borel subgroup $B\supset T$ of $G$. Let $W$ denote the Weyl group of $G$ with respect to $T$. From now on, we let $X=G/B$ denote the variety of Borel subgroups of $G$.

Let $\{g_w\}_{w\in W}$ be the Steinberg basis of $\K0(X)$ (cf. \cite{St75}). For each $w\in W$, $g_w$ is the class of a line bundle over $X$, and together they form a $\ZZ$-basis of $\K0(X)$. Combining this with the definition of the $\gamma$-filtration gives the following result. 

\begin{lem} The $i$-th term of the $\gamma$-filtration on $X=G/B$ is generated by products
$$
\gamma^{i/{i+1}}\K0(X)=\{\gamma_1(g_{w_1})\cdot \ldots\cdot \gamma_1(g_{w_i}) \mid w_1,\dots,w_i\in W\}.
$$ 
\end{lem}

Consider the twisted form $\XX$ of $X$ by means of a $G$-torsor $\xi\in Z^1(k,G)$.
In general the group $\K0(\XX)$ is not generated by classes of line bundles. Hence,
in the definition of $\gamma^i\K0(\XX)$, some higher characteristic classes $\gamma_{j}(-)$
may appear for $j>1$.

For every field extension $k'/k$ and variety $Y$ over $k$ we have a restriction map
$$
\resk\colon \K0(Y)\to \K0(Y\times_k k'). 
$$
For every extension $k'/k$, the map $\resk\colon \K0(X)\to \K0(X\times_k k')$ is an isomorphism. In particular, for a splitting field $l$ of $\xi$,  $\K0(\XX\times_k l) \simeq \K0(X\times_k l)\simeq \K0(X)$, so we may consider the restriction map
$$
\resl:\K0(\XX)\to\K0(X).
$$

The main result of \cite{Pa94} says that the image of $\resl$ coincides with the sublattice $$\langle i_{w,\xi} g_w \rangle_{w\in W},$$ where $i_{w,\xi}\ge 1$ are indices of the respective Tits algebras, which will be introduced in the next section. 

Note that characteristic classes commute with restrictions, that is,
$$
\gamma_j \circ \resk=\resk\circ \gamma_j\text{ and }c_j\circ \resk=\resk\circ c_j.
$$ 

We will use the following commutative diagram
$$
\xymatrix{
\igamtwist \ar[r]^{\resg} \ar[d]_{\cigam} & \igam \ar[d]^{\cigam}\\
\CH^i(\XX) \ar[r]^{\resch} & \CH^i(X).
}
$$

\begin{prop}  \label{propform}
Consider the composite 
$$\phi_i\colon \igamtwist\stackrel{\resg}\longrightarrow \igam\stackrel{\cigam}\longrightarrow \CH^i(X).$$
The image of $\phi_1$ is generated by 
$i_{w,\xi} c_1(g_w)$ for all $w\in W$. In general, the image of $\pi_i$ is generated by the elements 
$$(i-1)!\binom{i_{w_1}}{i_1}\cdots \binom{i_{w_m}}{i_m} c_1(g_{w_1})^{i_1}\cdots c_1(g_{w_m})^{i_m}$$ 
where $i_1+\dots +i_m=i$ for all $w_1,\dots,w_m\in W$. 
\end{prop}

\begin{proof} 
By the definitions of the restriction map and the $\gamma$-filtration on $\K0(X)$, we can see that the image of $\resg^{(i)}$ is generated by products
$$
\resg(\gamma^{i/{i+1}}\K0(\XX))=\{\gamma_{i_1}(i_{w_1} g_{w_1})\cdots\gamma_{i_m}(i_{w_m} g_{w_m})| i_1+\dots+i_m=i\},
$$
where $w_1,\dots,w_m\in W$.
\par
Since the $g_w$'s are line bundles, we have
\begin{equation} \label{bingamma}
\gamma(i_w g_w)=\gamma(g_w)^{i_w}=(1+\gamma_1(g_w)t)^{i_w}=\sum_{k=1}^{i_w}\binom{i_w}{k}\gamma_1(g_w)^kt^k
\end{equation}

Consider an element in $\im(\resg^{(i)})$ of the form $x=\gamma_{i_1}(i_{w_1} g_{w_1})\cdots\gamma_{i_m}(i_{w_m} g_{w_m})$ such that  $i_1+\dots+i_m=i$ for some  $w_1,\dots,w_m\in W$. By \eqref{bingamma}
$$
x=\gamma_{i_1}(i_{w_1} g_{w_1})\cdots\gamma_{i_m}(i_{w_m} g_{w_m})=\binom{i_{w_1}}{i_1}\cdots \binom{i_{w_m}}{i_m}\gamma_1(g_{w_1})^{i_1}\cdots\gamma_1(g_{w_m})^{i_m}.
$$
Taking the total Chern class and using Lemma \ref{lem:prelims} gives
\begin{align*}
c(x)&=c\Bigl(\binom{i_{w_1}}{i_1}\cdots \binom{i_{w_m}}{i_m}\gamma_1(g_{w_1})^{i_1}\cdots\gamma_1(g_{w_m})^{i_m}\Bigr)\\
		&=c(\gamma_1(g_{w_1})^{i_1}\cdots\gamma_1(g_{w_m})^{i_m})^{\binom{i_{w_1}}{i_1}\cdots \binom{i_{w_m}}{i_m}}\\
		&=(1+(-1)^{i-1}(i-1)!c_1(\gamma_1(g_{w_1}))^{i_1}\cdots c_1(\gamma_1((g_{w_m}))^{i_m}t^i)^{\binom{i_{w_1}}{i_1}\cdots \binom{i_{w_m}}{i_m}},
\end{align*}
and so by \eqref{gammatoc},
$$
c_i(x)=(-1)^{i-1}(i-1)!\binom{i_{w_1}}{i_1}\cdots \binom{i_{w_m}}{i_m} c_1(g_{w_1})^{i_1}\cdots c_1(g_{w_m})^{i_m}.
$$
\end{proof}


\section{Tits algebras and the \texorpdfstring{$J$}{J}-invariant}

Recall that we have defined $G$ to be a split linear algebraic group of rank $n$ over a field $k$. Also, we have fixed a split maximal torus $T\subset G$ and a Borel subgroup $B\supset T$. Let $T^*$ be the character group of $T$, $\{\alpha_1,\dots,\alpha_n\}$ a set of simple roots with respect to $B$ and $\{\omega_1,\dots,\omega_n\}$ the respective set of fundamental weights, so that $\alpha_i^\vee(\omega_j)=\delta_{ij}$. We have $\Lambda_r\subset T^*\subset \Lambda$, where $\Lambda_r$ is the root lattice and $\Lambda$ is the weight lattice. Consider the simply connected cover $\tilde{G}$ of $G$ with corresponding Borel subgroup $\tilde{B}$ and maximal split torus $\tilde{T}$. Given any $\lambda\in \Lambda=Hom(\tilde{T}, \mathbb{G}_m)$, we can lift $\lambda:\tilde{T}\rightarrow \mathbb{G}_m$ uniquely to $\lambda:\tilde{B}\rightarrow \mathbb{G}_m$. Letting 
$$
\tilde{G}\times^{\tilde{B}}V_1=\tilde{G}\times V_1/{(g,v)\sim (g\cdot b, \lambda(b)^{-1}\cdot v)},
$$ 
the projection map $\tilde{G}\times^{\tilde{B}}V_1\to \tilde{G}/\tilde{B}$ defines a line bundle $\mathcal{L}(\lambda)$ over $\tilde{G}/\tilde{B}=G/B$, the variety of Borel subgroups of $G$ (cf. \cite[\S1.5]{De74}). 
\par
For a fixed $\xi\in Z^1(k,G)$, we can associate to each weight $\lambda$ a central simple $k$-algebra $A_{\xi,\lambda}$, called a Tits algebra of $G$ (cf. \cite{Ti71}).  We define the Tits map 
$$
\beta_\xi:\Lambda/\Lambda_r\rightarrow Br(k)
$$
by sending $\bar{\lambda}\mapsto [A_{\xi,\bar{\lambda}}]$, its class in the Brauer group. This map is a group homomorphism for a fixed $\xi$, with $\bar{\lambda}_1+\bar{\lambda}_2\mapsto [A_{\xi,\bar{\lambda}_1}]\otimes[A_{\xi,\bar{\lambda}_2}]$. 

Consider again the variety $X=G/B$. The  degree 1 characteristic map in the simply connected case
$$
\cc_{sc}^{(1)}:\Lambda\tilde{\to} \CH^1(X)
$$
defines an isomorphism such that the cycles $h_i=c_1(\mathcal{L}(\omega_i))$, $i=1,\dots,n$ form a $\ZZ$-basis of $\CH^1(X)$. The degree 1 characteristic map is the restriction of this isomorphism to the character group $T^*$
$$
\cc^{(1)}:T^*\to \CH^1(X),
$$
mapping $\lambda=\sum_{i=1}^na_i\omega_i\mapsto c_1(\mathcal{L}(\lambda))=\sum_{i=1}^na_ih_i$. In general, this defines the characteristic map 
$$
\mathfrak{c}\colon S^*(T^*) \to \CH(X).
$$

We denote by $\pi:\CH^*(X)\to \CH^*(G)$ the pull-back induced by the natural projection $G\to X$. By \cite[Section 4, Rem. 2]{Gr58}, $\pi$ is surjective and its kernel is given by the ideal $I\subset \CH^*(X)$ generated by the non-constant elements in the image of the characteristic map. In particular, we have $I^{(1)}=\im(\cc^{(1)})$, and
$$
\CH^1(G)\simeq \CH^1(X)/(\im(\cc^{(1)}))\simeq \Lambda/T^*.
$$
Given a prime $p$, set $Ch(X)=\CH(X)\otimes\mathbb{F}_p$. Taking $\Fp$-coefficients, we have
$$
Ch^1(G)\simeq Ch^1(X)/(\im(\cc^{(1)}))\simeq \Lambda/T^*\otimes_\ZZ\Fp.
$$

It is known (cf. \cite{Kc85}) that $Ch(X)/I$ is isomorphic (as an $\mathbb{F}_p$-algebra, as well as a Hopf algebra) to a truncated polynomial ring of the form
$$
Ch(X)/I\cong \mathbb{F}_p[x_1,\dots,x_r]/(x_1^{p^{k_1}},\dots,x_r^{p^{k_r}})
$$
for some integers $r$ and $k_i\geq0$ for $i=1,\dots,r$, which are dependent on the group $G$. For each $i$, we let $d_i$ be the degree of the generator $x_i$. The number of generators of degree 1 is given by the dimension over $\Fp$ of the vector space $ \Lambda/T^*\otimes_\ZZ\Fp$. 
\par
Let  $s=dim_{\Fp}(Ch^1(G))$. Then, since $\omega_1,\dots,\omega_n$ generate $\Lambda$ we may choose a minimal set $\{i_1,\dots,i_s\}\subset \{1,\dots,n\}$ such that the classes of $\omega_{i_1},\dots,\omega_{i_s}$ generate $\Lambda/T^*\otimes\Fp$. Then, $h_{i_l}=c_1(\mathcal{L}(\omega_{i_l}))$, $l=1,\dots,s$ generate $Ch^1(X)$ and so we may take $x_l=\pi(h_{i_l})$, $l=1,\dots,s$ to be the generators of $Ch^1(G)$.
\par
In fact, this definition of the generators $x_1,\dots,x_s$ can be simplified using properties of the Steinberg basis. For $i=1,\dots,n$, $g_i:=\mathcal{L}(\omega_i)-\mathcal{L}(\alpha_i)$. But, $\Lambda_r\subset T^*$ implies that $c_1(\mathcal{L}(\alpha_i))\in \im(\cc^{(1)})$ and hence $c_1(\mathcal{L}(\alpha_i))\in ker(\pi)$ for all $i=1,\dots,n$. So, for each $l=1,\dots,s$, we have
$$
\pi(h_{i_l})=\pi(c_1(\mathcal{L}(\omega_{i_l})))=\pi(c_1(g_{i_l}))+\pi(c_1(\mathcal{L}(\alpha_{i_l})))=\pi(c_1(g_{i_l})).
$$
Thus, we may take the generators of $Ch^1(G)$ to be $x_l=\pi(c_1(g_{i_l}))$ for $l=1,\dots,s$. 

Let $H$ be the subgroup in $Br(k)$ generated by the classes of the Tits algebras $A_{\omega_{i_l}}=\beta_\xi(\bar{\omega}_{i_l})$ for $l=1,\dots,s$. Since $\Lambda/\Lambda_r$ is a finite abelian group, $H$ is a finite abelian group as well.
We define the {\em common index} $i_c$ of $\xi$ modulo $p$ by
$$
i_c:=\gcd\{\ind(A_{\omega_{i_1}}^{\otimes a_1}\otimes\cdots\otimes A_{\omega_{i_s}}^{\otimes a_s})\mid \text{at least one $a_l$ is coprime to $p$}\}.
$$

\begin{ex}
Let $G$ be a group of inner type $E_6$. In this case, $H$ is a cyclic group generated by a Tits algebra with index $3^d$ for $d=0,\dots,3$ \cite[6.4.1]{Ti71}. Therefore, by the definition of the common index, we have $i_c=3^d$ as well.
\end{ex}

We impose a well-ordering on the monomials $x_1^{m_1}\cdots x_r^{m_r}$ known as the \textit{DegLex} order \cite{PSZ}. For ease of notation, we denote the monomial $x_1^{m_1}\cdots x_r^{m_r}$ by $x^M$, where $M$ is the $r$-tuple of integers $(m_1,\dots,m_r)$, and set $|M|=\sum_{i=1}^rd_i m_i$. Given two $r$-tuples $M=(m_1,\dots,m_r)$ and $N=(n_1,\dots,n_r)$, we say that $x^M\leq x^N$ (or equivalently $M\leq N$) if either $|M|<|N|$, or $|M|=|N|$ and $m_i\leq n_i$ for the greatest $i$ such that $m_i\neq n_i$.
\par
Consider the restriction map $\resch\colon \CH^*(\XX)\to \CH^*(X)$.
Let $I_\xi$ denote the ideal generated by the non-constant elements in the image of $\resch$. In \cite[Thm. 6.4(1)]{KM06}, it is proven that  $I_\xi\supseteq I$ and that there always exists a $\xi$ over some field extension of $k$ such that $I_\xi=I$. Such a $\xi$ is called a ``generic" torsor.
\par
This inclusion induces surjections $\CH(X)/I\twoheadrightarrow \CH(X)/I_\xi$ and $Ch(X)/I\twoheadrightarrow Ch(X)/I_\xi$. For each $1\leq i\leq r$, we define $j_i$ to be the smallest integer such that $I_\xi$ contains an element $a$ of the form
$$
a=x_i^{p^{j_i}}+\sum_{x^M<x_i^{p^{j_i}}}c_Mx^M, \;\;\; c_M\in\Fp.
$$
While $r, d_i$ and $k_i$ for $i=1,\dots,r$ depend only on the group $G$, the values $j_1,\dots,j_r$ depend also on the choice of $\xi$. Thus given the \textit{DegLex} ordering defined above, we have a well-defined $r$-tuple $J_{p,\xi}(G)=(j_1,\dots,j_r)$, called the $J$-invariant of $G$. We note that $(0,\dots,0)\leq (j_1,\dots,j_r)\leq (k_1,\dots,k_r)$ for any choice of $\xi$. 
\par
Let $J_p^{(1)}=\{j_i \mid d_i=1\}$ be the subtuple of $J_{p,\xi}(G)$ consisting of only degree 1 parameters. We say that $J_p^{(1)}>m$ if for every index $j_l$ such that $k_l>m$ we have $j_l>m$.


\section{The main result}

For a fixed prime $p$, we have defined a minimal subset $\{\omega_{i_1},\dots,\omega_{i_s}\}\subset \Lambda$ such that the elements $x_l=\pi(c_1(g_{i_l}))$, $l=1,\dots,s$ generate $Ch^1(G)$. Let $I\subset \CH(X)$ be the ideal generated by the non-constant elements in the image of the characteristic map $\mathfrak{c}:S^*(T^*)\to \CH(X)$ and $I_\xi\subset \CH(X)$ be be the ideal generated by the non-constant elements in the image of the restriction map $\resch:\CH(\XX)\to \CH(X)$. 
For any integer $m$, we let $I^{(m)}\subset \CH^m(X)$ and $I_\xi^{(m)}\subset \CH^m(X)$ denote the homogeneous parts of these ideals of degree $m$. 

\begin{thm} \label{mainthm}
 If $v_p(i_c)>0$, then $I_\xi^{(1)}=I^{(1)}$. If $v_p(i_c)>1$, then $I_\xi^{(m)}=I^{(m)}$ for $m=2,\ldots, p$.
\end{thm}

\begin{proof}
Since we know already that $I\subseteq I_\xi$, it suffices to prove that $I_\xi^{(m)}\subseteq I^{(m)}$ for all $m=1,\dots,p$ under the relevant hypothesis on $i_c$. By Proposition \ref{propsurj} and the commutative diagram in Section 2, we have that for any $i\geq 0$,
$$
\im(\resch^{(i)})=\cigam(\im(\resg^{(i)})
$$
over the coefficient ring $\ZZ[\tfrac{1}{(i-1)!}]$.
\par
We begin first with the case $m=1$. By the definition of $I_\xi$, we have $I_\xi^{(1)}=\im(\resch^{(1)})$. To show that $I_\xi^{(1)}\subseteq I^{(1)}$, we must prove that if $v_p(i_c)>0$, then for any $w\in W$, the element $i_w c_1(g_w)$ belongs (after tensoring with $\mathbb{F}_p$) to $I^{(1)}=\im(\mathfrak{c}^{(1)})$. Recall that $g_w=\mathcal{L}(\rho_w)$, and that we may write $\rho_w=\sum_{i=1}^n a_i\omega_i$. Taking the total Chern class, we have
$$
c(g_w)=c(\mathcal{L}(\omega_1)^{\oplus a_1}\oplus\dots\oplus \mathcal{L}(\omega_n)^{\oplus a_n}) =1+\Bigl(\sum_{i=1}^n a_ic_1(\mathcal{L}(\omega_i))\Bigr)t+\dots,
$$
and hence, 
$$c_1(g_w)=\sum_{i=1}^n a_i c_1(\mathcal{L}(\omega_i))=\sum_{i=1}^n a_i (c_1(g_i)+c_1(\mathcal{L}(\alpha_i)))=\sum_{l=1}^sa_{i_l}c_1(g_{i_l})\text{        mod $\im(\mathfrak{c}^{(1)})$},
$$
by the results of Section 3.
If all $a_{i_l}\in\mathbb{Z}$ are divisible by $p$, we are done. So we assume at least one $a_{i_l}$ is coprime to $p$.

Applying the Tits map $\beta_\xi$ to the class of $\rho_w$, we get
$$
\beta_\xi(\bar{\rho}_w)=\beta_\xi\bigl(\sum_{l=1}^sa_{i_l}\bar{\omega}_{i_l}\bigr)=\bigotimes_{l=1}^s\beta_\xi(\bar{\omega}_{i_l})^{\otimes a_{i_l}}=\bigotimes_{l=1}^s[A_{\xi,\bar{\omega}_{i_l}}]^{\otimes a_{i_l}}
$$
By the assumptions that at least one of the $a_{i_l}$ is coprime to $p$ and that $v_p(i_c)>0$, we have $\beta_\xi(\bar{\rho}_w)\in H\setminus\{1\}$. Thus, $p|i_w=\ind(\beta_\xi(\bar{\rho}_w))$, and so $i_wc_1(g_w)=0$ in $Ch^1(X)$. 
\par
For the case $m>1$ we work under the hypothesis that $v_p(i_c)>1$, and proceed by induction.
We assume that the result $I_\xi^{(m')}\subseteq I^{(m')}$ holds for all $m'<m$. It can be seen that
$$
I_\xi^{(m)}=\Biggl(\bigoplus_{j=1}^{m-1}\CH^{m-j}(X)\cdot \im(\resch^{(j)})\Biggr)\oplus \im(\resch^{(m)}).
$$
By the inductive hypothesis, $\im(\resch^{(j)})\subseteq I_\xi^{(j)}\subseteq I^{(j)}$ for $1\leq j\leq m-1$, which implies that $\CH^{m-j}(X)\cdot \im(\resch^{(j)})\subseteq I^{(m)}$ for $1\leq j \leq m-1$. It remains to show that $\im(\resch^{(m)})\subseteq I^{(m)}$. 
\par
By Proposition \ref{propform} we know that $\im(\resch^{(m)})=c_m(\im(\resg^{(m)}))$, and is generated by elements of the form 
$$
a=(m-1)!\binom{i_{w_1}}{i_1}\cdots\binom{i_{w_k}}{i_k}c_1(g_{w_1})^{i_1}\cdots c_1(g_{w_k})^{i_k},
$$
where $i_1+\dots+i_k=m$, and $w_1,\dots,w_k\in W$. 
\par
If $i_l<m$ for all $l=1,\dots,k$, then $\binom{i_{w_l}}{i_l}c_1(g_{w_l})^{i_l}\in I^{(i_l)}$ by the inductive hypothesis. Therefore, $a\in I^{(i_1)}\cdots I^{(i_k)}\subseteq I^{(m)}$. If, on the other hand, $a$ is of the form $a=\binom{i_w}{m}c_1(g_w)^m$, then we apply the previous argument. Namely, we have
$$
\rho_w=\sum_{i=1}^na_i\omega_i,
$$
which implies 
$$
(c_1(g_w))^m=\Bigl(\sum_{i=1}^na_ic_1(\mathcal{L}(\omega_i))\Bigr)^m=\Bigl(\sum_{i=1}^na_i(c_1(g_i)+c_1(\mathcal{L}(\alpha_i)))\Bigr)^m.
$$
Again, $c_1(\mathcal{L}(\alpha_i))\in \im(\cc^{(1)})$ implies $c_1(\mathcal{L}(\alpha_i))\in I^{(1)}$ for all $i=1,\dots,n$, and so all terms in the above expansion that are divisible by some $c_1(\mathcal{L}(\alpha_i))$ are contained in $I^{(m)}$. As in the previous case, we may write
$$
(c_1(g_w))^m=\Bigl(\sum_{l=1}^sa_{i_l}c_1(g_{i_l})\Bigr)^m\text{    mod $I^{(m)}$}.
$$
If $a_{i_l}$ is divisible by $p$ for all $l=1,\dots,s$, then we are done, so we assume that at least one $a_{i_l}$ is coprime to $p$. As before, this ensures that $i_c\mid \ind(\beta_\xi(\bar{\rho}_w))$, and so $v_p(i_w)\ge v_p(i_c)$. It is clear that for any $b\in\mathbb{Z}_{>0}$ if $v_p(b)>1$ then $p \mid\binom{b}{l}$ for all $1\leq l \leq p$. Thus, under the hypothesis that $v_p(i_c)>1$, we have $p\mid \binom{i_w}{m}$, and so $\binom{i_w}{m}(c_1(g_w))^m=0$ in $Ch^m(X)$. 
\end{proof}

With this we obtain the following result, which can be seen as a generalization of Theorem 3.8 in \cite{QSZ}.

\begin{cor} \label{cor:main}
Let $p$ be a prime number. If $v_p(i_c)>0$, then $J^{(1)}_p> 0$.
If $v_p(i_c)>1$, then $J^{(1)}_p>1$.
\end{cor}

\begin{proof}
Consider the diagram
$$
\xymatrix{
Ch(\XX) \ar[r]^{\resch} & Ch(X) \ar[r]^\pi & Ch(G).
}
$$
We begin first with the hypothesis that $v_p(i_c)>0$.
\par
Let $R_\xi=\im(\pi\circ \resch)$. Then $\pi(a)\in R_\xi^{(1)}$ implies that $a\in \im(\resch^{(1)})=I^{(1)}$ by Theorem \ref{mainthm}. Thus, $\pi(a)=0\in Ch^1(X)/I$ and so $R_\xi^{(1)}=\{0\}$. Let $x_1,\dots,x_s$ be generators of degree 1 in $Ch^1(G)$. By the definition of the $J$-invariant, $j_1$ is the smallest non-negative integer $m$ such that $x_1^{p^m}\in R_\xi$. Since $x_1$ is non-trivial, we must have $x_1^{p^0}=x_1\not\in R_\xi^{(1)}$ by the above argument, and so $j_1>0$.
\par
The same argument applies for the remaining generators. Let $1< i\leq s$, then if $x^M<x_i$, $x^M=x_j$ for some $j<i$. Since $x_i+a_{i-1}x_{i-1}+\dots+a_1x_1$ is non-trivial for any $a_1,\dots,a_{i-1}\in \mathbb{F}_p$, it cannot belong to $R_\xi^{(1)}$. Therefore $j_i>0$, and so $J_p^{(1)}>0$.

Under the hypothesis that $v_p(i_c)>1$, suppose again that $x_1,\dots,x_s$ are a minimal set of generators of degree 1 in $Ch(G)$. We have the inclusion $\im(\resch^{(p)})\subset I_\xi^{(p)}$ and by Theorem \ref{mainthm}, $I_\xi^{(p)}=I^{(p)}$. Again, $R_\xi^{(p)}=\im(\pi\circ \resch^{(p)})=\{0\}$. To show that $J_p^{(1)}>1$, we begin with the generator $x_1$. If $k_1\leq 1$ we are done, so suppose $k_1>1$. Then, $x_1^{p^1}=x_1^p\in Ch^p(G)$ is non-trivial, and so $x^p\not\in R_\xi^{(p)}$ implies $j_1>1$. Again, we extend the argument for the remaining generators. Suppose that $k_i>1$ for some $1<i\leq s$. Then, the element 
$$\pi(a)=x_i^p+\sum_{(x^M<x_i^p)\cap (|M|=p)}a_Mx^M
$$ 
is non-trivial for any $a_M\in \Fp$ and hence $\pi(a)\notin R_\xi^{(p)}$, and so $j_i>1$. Thus $J_p^{(1)}>1$. 
\end{proof}


\section{Applications}

We now apply the results of the previous section to some $E_6$ varieties.
For this, we will require the following result concerning the possible values of the $J$-invariant.

\begin{lem} \label{lem:PS07}
Let $G$ be a semisimple algebraic group of inner type over $k$, $p$ a prime integer and $J_p(G) = (j_1,\dots,j_r)$. Assume $d_i=1$ for some $i=1,\dots,r$. Then $j_i\leq \max v_p(\ind A)$, where $A$ runs through all Tits algebras of $G$.
Conversely, if there exists a Tits algebra $A$ of $G$ with $v_p(\ind A)>0$, then $j_i>0$ for at least one $i$ having $d_i=1$.
\end{lem}

\begin{proof}
For the first statement, we note that for any $w\in W$, $c_1(\mathcal{L}(\rho_w))^{p^{b}}\in I_\xi$, where $b=v_p(i_{\rho_w})$ \cite[Lemma 1.12]{QSZ}. Since $\beta(\bar{\rho}_{s_{i_l}})=\beta(\bar{\omega}_{i_l})$ for $l=1, \dots,s$, this implies $i(\rho_{s_{i_l}})=i(\omega_{i_l})$. Letting $b_l=v_p(i_{\omega_{i_l}})$, we then have $c_1(g_{i_l})^{p^{b_l}}\in I_\xi$, and hence $j_l\leq b_l\leq \max v_p(\ind A)$ for all Tits algebras $A$ of $G$. For proof of the second statement, see \cite[Prop. 4.2]{PS09}. 
\end{proof}

\begin{ex}
Let $G$ be a simple group of type $B$, $C$, $E_6$, or $E_7$. Then, $H$ is a cyclic group of order $p=2$ or 3, and $J_p^{(1)}$ consists of a single integer $j_1$ \cite[Chap. 6]{Bou}. As a consequence of Corollary \ref{cor:main} and Lemma \ref{lem:PS07}, $j_1=0$ if and only if all of the Tits algebras of $G$ are split. 
\end{ex}

Let $G$ be a group of inner type $E_6$. As in the above example, $H$ is generated by a single Tits algebra $A$ of index $3^d$ for some $d=0,\ldots,3$.
Consider the $J$-invariant of $G$ modulo $p=3$. We note that $Ch(G)$ has precisely two generators $x_1$ and $x_2$, with $d_1=1$ and $d_2=4$, where the 3-power relations are defined by $k_1=2$ and $k_2=1$  \cite[Table II]{Kc85}. Thus
$$J_3(G)=(j_1,j_2),$$
where $j_1=0,1,2$, $j_2=0,1$. These values are independent of the characteristic of the base field.

\begin{prop}\label{maincor}
Let $G$ be a group of inner type $E_6$ and let $A$ be its Tits algebra. Then, $\ind A=1$ if and only if $j_1=0$. As well, $\ind A=3$ if and only if $j_1=1$.
\end{prop}

\begin{proof}
Since $A$ is a generator of the group $H$, the first case is a direct consequence of the above example. 
For the second case, suppose first that $j_1=1$. By Corollary \ref{cor:main} this implies that $\ind A=1$ or $3$. However, by the first case, $j_1\neq 0$ implies $\ind A\neq 1$ and so $\ind A=3$. Conversely, suppose $\ind A=3$. Then $v_3(\ind B)\leq 1$ for all $[B]\in H$. By Lemma \ref{lem:PS07}, $j_1\leq \max(v_3(\ind B))=1$. Again by the first case, $\ind A\neq 1$ implies $j_1\neq 0$ and so $j_1=1$. 
\end{proof}

\bibliographystyle{plain}

\begin{thebibliography}{10}

\bibitem{Bou} 
Bourbaki N., 
Groupes et alg\`ebres de Lie. Chap. 4, 5 et 6, 
Hermann, Paris 1968, 288pp. 

\bibitem{De74} 
Demazure, M.,  
D\'esingularisation des vari\'et\'es de Schubert g\'en\'eralis\'ees. 
Ann. Sci. \'Ecole Norm. Sup.~(4) {\bf 7} (1974), 
53--88. 

\bibitem{FL}
Fulton, W., Lang, S. 
Riemann-Roch algebra. Grundlehren der Math. Wiss. {\bf 277}, 
Springer-Verlag, New York, 1985, 
x+203pp.

\bibitem{Fu}
Fulton, W. 
Intersection theory. 2nd ed. 
Ergebnisse der Mathematik und ihrer Grenzgebiete. {\bf 3}. Folge. 
A Series of Modern Surveys in Mathematics {\bf 2}, 
Springer-Verlag, Berlin, 1998, xiv+470pp.

\bibitem{GPS}
Garibaldi, S., Semenov, N., Petrov, V.
Shells of twisted flag varieties and non-decomposibility of the Rost invariant.
arXiv:1012.2451v2 (2011), 48pp.

\bibitem{GZ10} 
Garibaldi, S., Zainoulline, K. 
The gamma-filtration and the Rost invariant. 
arXiv:1007.3482 (2010), 19pp.

\bibitem{Gr58}
Grothendieck, A. 
Torsion homologique et sections rationnelles in Anneaux de Chow et applications. 
S\'eminaire C.~Chevalley; 2e ann\'ee, 1958. 

\bibitem{Kc85}
Kac, V. 
Torsion in cohomology of compact Lie groups and Chow rings of reductive algebraic groups. 
Invent. Math. {\bf 80} (1985), 69--79.

\bibitem{Ka98}
Karpenko, N. 
Codimension $2$ cycles on Severi-Brauer varieties.  
$K$-Theory  {\bf 13 } (1998),  no.4, 305--330. 

\bibitem{KM06} Karpenko, N., Merkurjev, A.
Canonical $p$-dimension of algebraic groups.
Adv. Math. {\bf 205} (2006), 410--433. 

\bibitem{INV} Knus, M.-A., Merkurjev, A., Rost, M., Tignol, J.-P.
The Book of Involutions, AMS Colloquium Pub. {
bf 44}, Providence, RI, 1998.

\bibitem{MPW}
Merkurjev, A., Panin, I., Wadsworth, A. R. 
Index reduction formulas
for twisted flag varieties. I. 
$K$-Theory {\bf 10} (1996), 517--596.

\bibitem{Pa94}
Panin, I. 
On the Algebraic K-Theory of Twisted Flag Varieties. 
$K$-Theory {\bf 8}
(1994), 541--585.

\bibitem{PS09} 
Petrov, V., Semenov, N. 
Generically split projective homogeneous varieties. 
Duke Math. J. {\bf 152} (2010), 155--173.

\bibitem{PSZ}
Petrov, V., Semenov, N., Zainoulline, K.
$J$-invariant of linear algebraic groups.
Ann Scient. \'Ec. Norm. Sup. 4e s\'erie {\bf 41} (2008), 1023--1053.

\bibitem{QSZ} 
Qu\'eguiner-Mathieu, A., Semenov, N., Zainoulline, K. 
The J-invariant and the Tits algebras of a linear algebraic group. 
Preprint arXiv:1104.1096 (2011), pp.27.


\bibitem{PS07} 
Semenov, N. 
Higher Tits indices of algebraic groups (mit V.Petrov). 
Appendix von M.Florence. Preprint 2007

\bibitem{SGA6}
SGA 6, 
Th\'eorie des intersections et {T}h\'eor\`eme de {R}iemann-{R}och, 
Lecture Notes in Mathematics, vol. 225, Springer-Verlag, 1971.

\bibitem{St75}
Steinberg, R. 
On a Theorem of Pittie. 
Topology {\bf 14} (1975), 173--177.

\bibitem{Ti71}
Tits, J. 
Repr\'esentations lin\'eaires irr\'eductibles d'un groupe r\'eductif sur un corps quelconque. 
J. Reine Angew. Math. {\bf 247} (1971), 196--220.

\end{thebibliography}

\end{document}